\documentclass[12pt]{amsart}
\usepackage[ruled,vlined]{algorithm2e}
\usepackage{fancyhdr,graphicx,amsmath,amssymb,amsthm}
\usepackage[left=1.0in,right=1.0in,bottom=1.0in]{geometry}
\usepackage{tcolorbox, enumitem, verbatim}
\usepackage{yfonts}
\usepackage{hyperref}
\usepackage{tikz}
\usepackage[utf8]{inputenc}
\usepackage{mathtools}
\usepackage[mathscr]{euscript}
\usepackage{faktor}
\usepackage{quiver}
\usepackage{tikz-cd}
\usepackage{bbm}
\usepackage{faktor}
\usepackage{mathabx}
\usepackage{thmtools}
\usepackage[backend=biber, style=alphabetic]{biblatex}

\theoremstyle{plain}
\newtheorem{theorem}{Theorem}
\newtheorem{coroll}[theorem]{Corollary}

\newtheorem{prop}[theorem]{Proposition}

\newtheorem{claim}[theorem]{Claim}

\theoremstyle{definition}

\newtheorem*{definition*}{Definition}

\theoremstyle{remark}
\newcommand{\CC}{\mathbb{C}}
\newcommand{\ZZ}{\mathbb{Z}}

\newcommand{\sh}[1]{\mathcal{#1}}
\newcommand{\Rm}[1]{\mathrm{#1}}

\newcommand{\Spec}{\operatorname{Spec}}

\newcommand{\tab}[1][0.5cm]{\hspace*{#1}}

\newcommand{\genlegendre}[4]{%
  \genfrac{(}{)}{}{#1}{#3}{#4}%
  \if\relax\detokenize{#2}\relax\else_{\!#2}\fi
}

\DeclarePairedDelimiter{\set}{\{}{\}}

\DeclarePairedDelimiter{\brac}{(}{)}

\DeclareMathAlphabet{\pazocal}{OMS}{zplm}{m}{n}

\newcommand{\Comment}[1]{}

\def\XXint#1#2#3{{\setbox0=\hbox{$#1{#2#3}{\int}$}
     \vcenter{\hbox{$#2#3$}}\kern-.5\wd0}}

\title[On the NAHC for higher-dimensional quasiprojective varieties]{On the nilpotent residue non-abelian Hodge correspondence for higher-dimensional quasiprojective varieties}
\author{Quoc Anh Tran}
\address{A. Tran: Dept. of Mathematics, Statistics, and Computer Science, University of
Illinois at Chicago, Chicago, USA}
\email{atran58@uic.edu}
\date{\today}

\begin{document}
\begin{abstract}
    In \cite{BBT24}, the authors proved that on a projective log smooth variety $\brac*{\bar{X}, D}$ there is a continuous bijection between the moduli space $M^{\Rm{nilp}}_{\Rm{Dol}}\brac*{\bar{X}, D}$ of logarithmic Higgs bundles with nilpotent residues and the moduli space $M^{\Rm{nilp}}_{\Rm{DR}}\brac*{\bar{X}, D}$ of logarithmic connections with nilpotent residues. In this notes, we argue that the map is a homeomorphism. 
\end{abstract}
\maketitle

Let $\brac*{\bar{X}, D}$ be a projective log smooth complex variety, i.e., a projective smooth complex variety $\bar{X}$ with a simple normal crossing divisor $D$. In the sense of \cite{BBT24}, we have two good moduli spaces:
\begin{align*}
    M^{\Rm{nilp}}_{\Rm{DR}}\brac*{\bar{X}, D, r} &= \left\{ (E, \nabla) \;\middle|\; 
    \begin{aligned}
        & (E, \nabla) \text{ is a rank } r \text{ polystable log connection} \\
        & \text{with nilpotent residues}
    \end{aligned}
    \right\} \\[1em]
    M^{\Rm{nilp}}_{\Rm{Dol}}\brac*{\bar{X}, D, r} &= \left\{ (E, \theta) \;\middle|\; 
    \begin{aligned}
        & (E, \theta) \text{ is a rank } r \text{ polystable log Higgs bundle} \\
        & \text{with nilpotent residues and } c_i(E) = 0 \text{ for all }i
    \end{aligned}
    \right\}
\end{align*}
and as a corollary of \cite[Theorem 1.1]{TM09}, we get a bijective comparison map: 
\[SM_{\brac*{\bar{X}, D}}^{\Rm{nilp}}: M^{\Rm{nilp}}_{\Rm{Dol}}\brac*{\bar{X}, D, r}(\CC) \to M^{\Rm{nilp}}_{\Rm{DR}}\brac*{\bar{X}, D, r}(\CC)\]
which is a homeomorphism, in the Euclidean topology and when $\bar{X}$ is a curve, by \cite[Theorem 7.13]{BBT24}. They were also able to prove that $SM_{\brac*{\bar{X}, D}}^{\Rm{nilp}}$ is a continuous bijection in higher dimensions, as follows. 

\begin{definition*}
  A Lefschetz curve is a smooth complete intersection curve $\iota: \bar{C} \to \bar{X}$ such that $\bar{C}$ is transverse to $D$, and we have a surjection $\pi_1\brac*{\bar{C} \backslash D} \twoheadrightarrow \pi_1\brac*{\bar{X} \backslash D}$. 
\end{definition*}

Such a Lefschetz curve $\iota: \bar{C} \hookrightarrow \bar{X}$ induces a closed embedding $i^*_{\Rm{DR}}: M^{\Rm{nilp}}_{\Rm{DR}}\brac*{\bar{X}, D, r} \to M^{\Rm{nilp}}_{\Rm{DR}}\brac*{\bar{C}, D \cap \bar{C}, r}$
and a diagram:
\[\begin{tikzcd}
	{M^{\Rm{nilp}}_{\Rm{Dol}}\brac*{\bar{X}, D, r}} && {M^{\Rm{nilp}}_{\Rm{Dol}}\brac*{\bar{C}, D \cap \bar{C}, r}} \\
	\\
	{M^{\Rm{nilp}}_{\Rm{DR}}\brac*{\bar{X}, D, r}} && {M^{\Rm{nilp}}_{\Rm{DR}}\brac*{\bar{C}, D \cap \bar{C}, r}}
	\arrow["{i_{\Rm{Dol}}^*}", from=1-1, to=1-3]
	\arrow["{SM_{\brac*{\bar{X}, D}}^{\Rm{nilp}}}"', from=1-1, to=3-1]
	\arrow["{SM_{\brac*{\bar{C}, D \cap \bar{C}}}^{\Rm{nilp}}}", from=1-3, to=3-3]
	\arrow["{i_{\Rm{DR}}^*}"', from=3-1, to=3-3]
\end{tikzcd}\]
\tab Since $SM_{\brac*{\bar{C}, D \cap \bar{C}}}^{\Rm{nilp}}$ is a homeomorphism and $i_{\Rm{DR}}^*$ is a closed immersion, $SM_{\brac*{\bar{X}, D}}^{\Rm{nilp}}$ is continuous, thus a continuous bijection as desired. Our main result is the following:
\begin{theorem} 
  There exists a Lefschetz curve $\bar{C} \subset \bar{X}$ so that the restriction map $i_{\Rm{Dol}}^*$ is proper.
\end{theorem}
Then, by running the previous argument for the inverse comparison maps, we get: 
\begin{coroll}
  $SM_{\brac*{\bar{X}, D}}^{\Rm{nilp}}$ is a homeomorphism.
\end{coroll}
To get to the proof of theorem 1, we will need some preliminary results and definitions:
\begin{definition*}
  A rank $r$ logarithmic Higgs bundle on $\brac*{\bar{X}, D}$ is a rank $r$ locally free $\sh{O}_{\bar{X}}-$module $E$ together with a $\sh{O}_{\bar{X}}-$linear morphism $\theta: E \to \Omega_{\bar{X}}(\log D) \otimes_{\sh{O}_{\bar{X}}} E$ such that $\theta \wedge \theta = 0$, called the Higgs field. Locally, $\theta$ is just a matrix of logarithmic 1-forms, hence we can talk about its characteristic polynomial $\Rm{char}(\theta)$. The coefficients are then in $\bigoplus_{i = 0}^r H^0\brac*{\bar{X}, \Rm{Sym}^i \Omega_{\bar{X}}(\log D)}$, and we get the Hitchin map: 
  \[M^{\Rm{nilp}}_{\Rm{Dol}}\brac*{\bar{X}, D, r} \xrightarrow{h_{\brac*{\bar{X}, D}}} \bigoplus_{i = 0}^r H^0\brac*{\bar{X}, \Rm{Sym}^i \Omega_{\bar{X}}(\log D)}\]
  \tab Now if $E$ is only torsion-free, then there is an open $U \subseteq \bar{X}$, with $\Rm{codim}_{\bar{X}}\brac*{\bar{X} - U} \geq 2$, such that $E\vert_U$ is locally free. We can then talk about the characteristic polynomial of $(E\vert_U, \theta\vert_U)$. Since $\bar{X}$ is smooth (in particular, normal), the coefficients (over $U$) extends uniquely to sections over $\bar{X}$. Hence, if we define $M_{\Rm{Dol}}^{\Rm{tf}}\brac*{\bar{X}, D, P}$ to be the moduli space of Gieseker polystable torsion-free logarithmic Higgs sheaves with Hilbert polynomial $P$, then we have another Hitchin map with a bigger total space:
  \[M^{\Rm{tf}}_{\Rm{Dol}}\brac*{\bar{X}, D, P} \xrightarrow{h'_{\brac*{\bar{X}, D, P}}} \bigoplus_{i = 0}^r H^0\brac*{\bar{X}, \Rm{Sym}^i \Omega_{\bar{X}}(\log D)}\]
\end{definition*}

\begin{prop}{\cite[Theorem 3.8]{AL14}}
  The Hitchin map $h_{\brac*{\bar{X}, D}}$ is proper.
\end{prop}
\begin{proof}
  Let us say a few words about why this is implied by \cite[Theorem 3.8]{AL14}. That the torsion-free Hitchin map $h'_{\brac*{\bar{X}, D, P}}$ is proper is just a special case of the cited theorem. Indeed, a logarithmic Higgs sheaf is just a $T_{\bar{X}}(-\log D)-$coHiggs sheaf, viewing $T_{\bar{X}}(-\log D)$ as a trivial Lie algebroid. \\
  \tab $M^{\Rm{nilp}}_{\Rm{Dol}}\brac*{\bar{X}, D, r}$ is the subscheme of $M^{\Rm{tf}}_{\Rm{Dol}}\brac*{\bar{X}, D, r\cdot P_{\sh{O}_{\bar{X}}}}$, where $P_{\sh{O}_{\bar{X}}}$ is the Hilbert polynomial of $\sh{O}_{\bar{X}}$, cutout by requiring all Chern classes to vanish (which forces the torsion-free sheaves to be locally free) and then requiring the residues of the Higgs fields to be nilpotent. If we can show $M^{\Rm{nilp}}_{\Rm{Dol}}\brac*{\bar{X}, D, r}$ is a closed subscheme, then $h_{\brac*{\bar{X}, D}}$, which is the restriction of a proper map to a closed subset, must be proper as desired. \\
  \tab Now, Chern classes are locally constant in a flat family: consider torsion-free $\sh{E}$ on $\bar{X} \times S$ with two fibers $\bar{X}_s \xrightarrow{\iota_s} \bar{X} \times S, \bar{X}_t \xrightarrow{\iota_t} \bar{X} \times S$, we claim that if $S$ is connected then $c_i(\iota_s^*\sh{E}) \simeq c_i(\iota_t^*\sh{E}) \in H^{2i}\brac*{\bar{X}, \ZZ}$. By naturality of Chern classes, these are just $\iota_s^*c_i(\sh{E})$ and $\iota_t^*c_i(\sh{E})$. Since $S$ connected, after identifying $\bar{X}_s \simeq \bar{X}_t \simeq \bar{X}$ it's easy to see $\iota_s$ and $\iota_t$ are homotopic (via a path $\gamma$ connecting $s, t \in S$) hence the two pullbacks must agree. Thus requiring vanishing Chern classes pickout a union of connected components of $M^{\Rm{tf}}_{\Rm{Dol}}\brac*{\bar{X}, D, r\cdot P_{\sh{O}_{\bar{X}}}}$, hence closed. Further requiring nilpotent residue is also a closed condition, see \cite[p. 66]{BBT24}. 
\end{proof}
Next, we have a commutative diagram:
\[\begin{tikzcd}
	{M^{\Rm{nilp}}_{\Rm{Dol}}\brac*{\bar{X}, D, r}} && {\bigoplus_{i = 0}^r H^0\brac*{\bar{X}, \Rm{Sym}^i \Omega_{\bar{X}}(\log D)}} \\
	\\
	{M^{\Rm{nilp}}_{\Rm{Dol}}\brac*{\bar{C}, D \cap \bar{C}, r}} && {\bigoplus_{i = 0}^r H^0\brac*{\bar{C}, \Rm{Sym}^i \Omega_{\bar{C}}(\log D \cap \bar{C})}}
	\arrow["{h_{\brac*{\bar{X}, D}}}", from=1-1, to=1-3]
	\arrow["{i_{\Rm{Dol}}^*}"', from=1-1, to=3-1]
	\arrow["{\Rm{res}}", from=1-3, to=3-3]
	\arrow["{h_{\brac*{\bar{C}, D \cap \bar{C}}}}"', from=3-1, to=3-3]
\end{tikzcd}\]
where the right vertical arrow is the composition 
\[H^0\brac*{\bar{X}, \Rm{Sym}^i \Omega_{\bar{X}}(\log D)} \to H^0\brac*{\bar{C}, \Rm{Sym}^i \Omega_{\bar{X}}(\log D)\vert_{\bar{C}}} \to H^0\brac*{\bar{C}, \Rm{Sym}^i \Omega_{\bar{C}}(\log D \cap \bar{C})}\]
\begin{prop}
  There exists a Lefschetz curve $\bar{C}$ so that $\Rm{res}$ is injective.
\end{prop}
\begin{proof}
  Let's first examine what it means to be in the kernel of the map
  \[H^0\brac*{\bar{X}, \Rm{Sym}^i \Omega_{\bar{X}}(\log D)} \xrightarrow{\Rm{res}} H^0\brac*{\bar{C}, \Rm{Sym}^i \Omega_{\bar{C}}(\log D \cap \bar{C})}\]
  \tab Suppose that $s \in H^0\brac*{\bar{X}, \Rm{Sym}^i \Omega_{\bar{X}}(\log D)}$ maps to the zero section under $\Rm{res}$. At each point $p \in \bar{C}, p \notin D$, if we restrict $s$ to the fibers at $p$, we can view $s(p) \in \Rm{Sym}^i \brac*{T_p\bar{X}}^\vee$ as a homogeneous polynomial on $T_p\bar{X}$. The map $\Rm{res}$ then just restricts $s(p)$ to the line $T_p\bar{C} \subset T_p\bar{X}$. As a result, $\Rm{res}(s) = 0$ implies that:
  \[s(p)\brac*{T_p\bar{C}} = 0 \tab \forall\ p \in \bar{C}, p \notin D\]
  \tab Thus it suffices to prove that there exists a Lefschetz curve $\bar{C}$ such that for each $s \in H^0\brac*{\bar{X}, \Rm{Sym}^i \Omega_{\bar{X}}(\log D)}$, there exists $p \in \bar{C}, p \notin D$ such that $T_p\bar{C} \not\subset V(s(p)) \subset T_p\bar{X}$. We first prove the following claim: 
  \begin{claim}
    There exists a finite number of pairs $\set*{(p_j, l_j)}_{j = 1}^t$ with $p_j \in \bar{X}, p_j \notin D$, $l_j$ a general 1-dimensional $\CC-$subspace of $T_{p_j}\bar{X}$ such that for every $s \in H^0\brac*{\bar{X}, \Rm{Sym}^i \Omega_{\bar{X}}(\log D)}$, 
    \[s(p_j)(l_j) \neq 0\]
    for at least one of the pairs. 
  \end{claim}
  \begin{proof}[Proof of claim]
    Let $s_1 \in H^0\brac*{\bar{X}, \Rm{Sym}^i \Omega_{\bar{X}}(\log D)}$ be a nonzero section, then there is some point $p_1 \in \bar{X}, p_1 \notin D$ such that $s_1(p_1) \in \Rm{Sym}^i\brac*{T_{p_1}\bar{X}}^\vee$ is nonzero. This is a polynomial on $T_{p_1}\bar{X}$, so the zero locus $V(s_1(p_1))$ is a hypersurface. Thus we can just pick a line $l_1$ not contained in $V(s_1(p_1))$. The pair $(p_1, l_1)$ gives a linear functional:
    \[\epsilon_1: H^0\brac*{\bar{X}, \Rm{Sym}^i \Omega_{\bar{X}}(\log D)} \to \Rm{Sym}^i\brac*{l_1^\vee} \simeq \CC, \tab s \mapsto s(p_1)\vert_{l_1}\]
    and $\epsilon_1(s_1) \neq 0$. So $\dim \ker(\epsilon_1) < \dim H^0\brac*{\bar{X}, \Rm{Sym}^i \Omega_{\bar{X}}(\log D)}$. If $\ker(\epsilon_1) = 0$ we are done, otherwise pick $s_2 \in \ker(\epsilon_1)$ and repeat the process. Since $H^0\brac*{\bar{X}, \Rm{Sym}^i \Omega_{\bar{X}}(\log D)}$ is finite dimensional, we end up with a finite collection $\epsilon_1, \dots, \epsilon_t$ such that 
    \[\bigcap_{j = 1}^t \ker(\epsilon_j) = 0\]
    and the corresponding pairs $(p_j, l_j)$ satisfy the claim. 
  \end{proof}
  Our proposition is reduced to showing that there exists a Lefschetz curve $\bar{C}$ which:
    \begin{itemize}
      \item passes through the points $\set*{p_j}_{j = 1}^t$;
      \item and for each $j$, $T_{p_j}\bar{C} = l_j$. 
    \end{itemize}
  \tab Up to picking a suitable $\sh{O}_{\bar{X}}(1)$, this follows from \cite[Theorem A.1]{achter2019parameter} part (3) and (4). Indeed, their theorem implies that we can pick a general linear section $\bar{C}$ which is smooth, and passes through $\set*{p_j}$ with prescribed tangents. Now, being transverse to $D$ and inducing surjection of fundamental groups are also open conditions (by \cite[Theorem 1.1.3(ii)]{hamm1985lefschetz}), so there exists a Lefschetz choice. 
\end{proof}
We are finally ready for the main result:
\begin{proof}[Proof of Theorem 1]
  We will use the valuative criterion for properness: given a discrete valuation ring $R$, with $K = \Rm{Frac}(R)$, and a diagram
  \[\begin{tikzcd}
	{\Spec K} && {M^{\Rm{nilp}}_{\Rm{Dol}}\brac*{\bar{X}, D, r}} \\
	\\
	{\Spec R} && {M^{\Rm{nilp}}_{\Rm{Dol}}\brac*{\bar{C}, D \cap \bar{C}, r}}
	\arrow[from=1-1, to=1-3]
	\arrow[from=1-1, to=3-1]
	\arrow["{{i_{\Rm{Dol}}^*}}"', from=1-3, to=3-3]
	\arrow[dashed, from=3-1, to=1-3]
	\arrow[from=3-1, to=3-3]
  \end{tikzcd}\]
  we want to show that the dotted arrow exists. Consider the diagram
  \[\begin{tikzcd}
	{\Spec K} && {M^{\Rm{nilp}}_{\Rm{Dol}}\brac*{\bar{X}, D, r}} && {\bigoplus_{i = 0}^r H^0\brac*{\bar{X}, \Rm{Sym}^i \Omega_{\bar{X}}(\log D)}} \\
	\\
	{\Spec R} && {M^{\Rm{nilp}}_{\Rm{Dol}}\brac*{\bar{C}, D \cap \bar{C}, r}} && {\bigoplus_{i = 0}^r H^0\brac*{\bar{C}, \Rm{Sym}^i \Omega_{\bar{C}}(\log D \cap \bar{C})}}
	\arrow[from=1-1, to=1-3]
	\arrow[from=1-1, to=3-1]
	\arrow["{{h_{\brac*{\bar{X}, D}}}}", from=1-3, to=1-5]
	\arrow["{{i_{\Rm{Dol}}^*}}"', from=1-3, to=3-3]
	\arrow["{{\Rm{res}}}", from=1-5, to=3-5]
	\arrow[dashed, from=3-1, to=1-3]
	\arrow[from=3-1, to=1-5]
	\arrow[from=3-1, to=3-3]
	\arrow["{{h_{\brac*{C, D \cap C}}}}"', from=3-3, to=3-5]
  \end{tikzcd}\]
  \tab The long diagonal arrow exists (and makes the diagram commute) since $\Rm{res}$ is an injective map of vector spaces (equivalently, a closed embedding of affine spaces). Then the dotted arrow must exist since $h_{\brac*{\bar{X}, D}}$ is proper. Finally, it's a trivial diagram chase to check that the left square commutes.
\end{proof}

\subsection*{Acknowledgement} I would like to thank my advisor, Benjamin Bakker, for his patience and support, as well as numerous helpful conversations. 

\printbibliography
\end{document}